\documentclass[12pt]{amsart}
\usepackage{amsmath,amsthm,amsfonts,amssymb,mathrsfs}
\date{\today}

\usepackage{color}
\input xymatrix
\xyoption{all}

\usepackage{color}

\usepackage{amssymb,cite}
\usepackage[colorlinks,plainpages,citecolor=magenta, linkcolor=blue, backref]{hyperref}%,bookmarksnumbered,

\usepackage{hyperref}

\newtheorem{theorem}{Theorem}%[section]

\newtheorem{proposition}[theorem]{Proposition}
\newtheorem{corollary}[theorem]{Corollary}
\newtheorem{lemma}[theorem]{Lemma}
\theoremstyle{definition}
\newtheorem{example}[theorem]{Example}%[section]
\newtheorem{remark}[theorem]{Remark}%[section]

\newtheorem{definition}[theorem]{Definition}%[section]

\begin{document}

\title[A note on feebly compact semitopological symmetric inverse semigroups of ...]{A note on feebly compact semitopological symmetric inverse semigroups of a bounded finite rank}

\author{Oleg~Gutik}
\address{Faculty of Mathematics, National University of Lviv,
Universytetska 1, Lviv, 79000, Ukraine}
\email{oleg.gutik@lnu.edu.ua}



\begin{abstract}
We study feebly compact shift-continuous $T_1$-topologies on the symmetric inverse semigroup $\mathscr{I}_\lambda^n$ of finite transformations of the rank $\leqslant n$. It is proved that such $T_1$-topology is sequentially pracompact if and only if it is feebly compact. Also, we show that every shift-continuous feebly $\omega$-bounded $T_1$-topology on $\mathscr{I}_\lambda^n$ is compact.
\end{abstract}

\maketitle

\section{Introduction and preliminaries}

We follow the terminology of the monographs~\cite{Carruth-Hildebrant-Koch-1983-1986, Clifford-Preston-1961-1967, Engelking-1989, Lawson-1998, Petrich-1984, Ruppert-1984}. If $X$ is a topological space and $A\subseteq X$, then by $\operatorname{cl}_X(A)$ and $\operatorname{int}_X(A)$ we denote the topological closure and interior of $A$ in $X$, respectively. By $|A|$ we denote the cardinality of a set $A$, by $A\triangle B$ the symmetric difference of sets $A$ and $B$, by $\mathbb{N}$ the set of positive integers, and by $\omega$ the first infinite cardinal.  By $\mathfrak{D}(\omega)$ and $\mathbb{R}$ we denote an infinite countable discrete space and the real numbers with the usual topology, respectively.

A semigroup $S$ is called \emph{inverse} if every $a$ in $S$ possesses an unique inverse $a^{-1}$, i.e. if there exists an unique element $a^{-1}$ in $S$ such that
\begin{equation*}
    aa^{-1}a=a \qquad \mbox{and} \qquad a^{-1}aa^{-1}=a^{-1}.
\end{equation*}
A map which associates to any element of an inverse semigroup its inverse is called the \emph{inversion}.

If $S$ is a~semigroup, then by $E(S)$ we denote the subset of all idempotents of $S$. On  the set of idempotents $E(S)$ there exists a natural partial order: $e\leqslant f$ \emph{if and only if} $ef=fe=e$. A \emph{semilattice} is a commutative semigroup of idempotents. We observe that the set of idempotents of an inverse semigroup is a semilattice~\cite{Wagner-1952}.

Every inverse semigroup $S$ admits a partial order:
\begin{equation*}
  a\preccurlyeq b \qquad \hbox{if and only if there exists} \qquad e\in E(S) \quad \hbox{such that} \quad a=eb.
\end{equation*}
We shall say that $\preccurlyeq$ is the \emph{natural partial order} on $S$ (see \cite{Carruth-Hildebrant-Koch-1983-1986, Wagner-1952}).

Let $\lambda$ be an arbitrary nonzero cardinal. A map $\alpha$ from a subset $D$ of $\lambda$ into $\lambda$ is called a \emph{partial transformation} of $\lambda$. In this case the set $D$ is called the \emph{domain} of $\alpha$ and is denoted by $\operatorname{dom}\alpha$. The image of an element $x\in\operatorname{dom}\alpha$ under $\alpha$ is denoted by $x\alpha$.  Also, the set $\{ x\in \lambda\colon y\alpha=x \mbox{ for some } y\in Y\}$ is called the \emph{range} of $\alpha$ and is denoted by $\operatorname{ran}\alpha$. For convenience we denote by $\varnothing$ the empty transformation, a partial mapping with $\operatorname{dom}\varnothing=\operatorname{ran}\varnothing=\varnothing$.

Let $\mathscr{I}_\lambda$ denote the set of all partial one-to-one transformations of $\lambda$ together with the following semigroup operation:
\begin{equation*}
    x(\alpha\beta)=(x\alpha)\beta \quad \mbox{if} \quad
    x\in\operatorname{dom}(\alpha\beta)=\{
    y\in\operatorname{dom}\alpha\colon
    y\alpha\in\operatorname{dom}\beta\}, \qquad \mbox{for} \quad
    \alpha,\beta\in\mathscr{I}_\lambda.
\end{equation*}
The semigroup $\mathscr{I}_\lambda$ is called the \emph{symmetric
inverse semigroup} over the cardinal $\lambda$~(see \cite{Clifford-Preston-1961-1967}). For any $\alpha\in\mathscr{I}_\lambda$ the cardinality of $\operatorname{dom}\alpha$ is called the \emph{rank} of $\alpha$ and it is denoted by $\operatorname{rank}\alpha$. The symmetric inverse semigroup was introduced by V.~V.~Wagner~\cite{Wagner-1952}
and it plays a major role in the theory of semigroups.

%For every $\alpha\in \mathscr{I}_\lambda$ we put $\operatorname{rank}\alpha=\left|\operatorname{dom}\alpha\right|$.

Put
$\mathscr{I}_\lambda^n=\{ \alpha\in\mathscr{I}_\lambda\colon
\operatorname{rank}\alpha\leqslant n\}$,
for $n=1,2,3,\ldots$. Obviously,
$\mathscr{I}_\lambda^n$ ($n=1,2,3,\ldots$) are inverse semigroups,
$\mathscr{I}_\lambda^n$ is an ideal of $\mathscr{I}_\lambda$, for each $n=1,2,3,\ldots$. The semigroup
$\mathscr{I}_\lambda^n$ is called the \emph{symmetric inverse semigroup of
finite transformations of the rank $\leqslant n$} \cite{Gutik-Reiter-2009}. By
\begin{equation*}
\left({%
\begin{array}{cccc}
  x_1 & x_2 & \cdots & x_n \\
  y_1 & y_2 & \cdots & y_n \\
\end{array}%
}\right)
\end{equation*}
we denote a partial one-to-one transformation which maps $x_1$ onto $y_1$, $x_2$ onto $y_2$, $\ldots$, and $x_n$ onto $y_n$. Obviously, in such case we have $x_i\neq x_j$ and $y_i\neq y_j$ for $i\neq j$ ($i,j=1,2,3,\ldots,n$). The empty partial map $\varnothing\colon \lambda\rightharpoonup\lambda$ is denoted by $\boldsymbol{0}$. It is obvious that $\boldsymbol{0}$ is zero of the semigroup $\mathscr{I}_\lambda^n$.

Let $\lambda$ be a nonzero cardinal. On the set
 $
 B_{\lambda}=(\lambda\times\lambda)\cup\{ 0\}
 $,
where $0\notin\lambda\times\lambda$, we define the semigroup
operation ``$\, \cdot\, $'' as follows
\begin{equation*}
(a, b)\cdot(c, d)=
\left\{
  \begin{array}{cl}
    (a, d), & \hbox{ if~ } b=c;\\
    0, & \hbox{ if~ } b\neq c,
  \end{array}
\right.
\end{equation*}
and $(a, b)\cdot 0=0\cdot(a, b)=0\cdot 0=0$ for $a,b,c,d\in
\lambda$. The semigroup $B_{\lambda}$ is called the
\emph{semigroup of $\lambda\times\lambda$-matrix units}~(see
\cite{Clifford-Preston-1961-1967}). Obviously, for any cardinal $\lambda>0$, the semigroup
of $\lambda\times\lambda$-matrix units $B_{\lambda}$ is isomorphic
to $\mathscr{I}_\lambda^1$.

A subset $A$ of a topological space $X$ is called \emph{regular open} if $\operatorname{int}_X(\operatorname{cl}_X(A))=A$.

We recall that a topological space $X$ is said to be
\begin{itemize}
  \item \emph{semiregular} if $X$ has a base consisting of regular open subsets;
  \item \emph{compact} if each open cover of $X$ has a finite subcover;
  \item \emph{sequentially compact} if each sequence $\{x_i\}_{i\in\mathbb{N}}$ of $X$ has a convergent subsequence in $X$;
  \item \emph{countably compact} if each open countable cover of $X$ has a finite subcover;
  \item \emph{H-closed} if $X$ is a closed subspace of every Hausdorff topological space in which it is contained;
  \item \emph{$\omega$-bounded-pracompact} if $X$ contains a dense subset $D$ such that each countable subset of $D$ has the compact closure in $X$ \cite{Gutik-Ravsky-20??};
  \item \emph{infra H-closed} provided that any continuous image of $X$ into any first countable Hausdorff space is closed (see \cite{Hajek-Todd-1975});
  \item \emph{totally countably pracompact} if there exists a dense subset $D$ of the space $X$ such that each sequence of points of the set $D$ has a subsequence with the compact closure in $X$ \cite{Gutik-Ravsky-20??};
  \item \emph{sequentially pracompact} if there exists a dense subset $D$ of the space $X$ such that each sequence of points of the set $D$ has a convergent subsequence \cite{Gutik-Ravsky-20??};
  \item \emph{countably compact at a subset} $A\subseteq X$ if every infinite subset $B\subseteq A$  has  an  accumulation  point $x$ in $X$ \cite{Arkhangelskii-1992};
  \item \emph{countably pracompact} if there exists a dense subset $A$ in $X$  such that $X$ is countably compact at $A$ \cite{Arkhangelskii-1992};
  \item \emph{feebly $\omega$-bounded} if for each sequence $\{U_n\}_{n\in\mathbb{N}}$
  of nonempty  open  subsets of $X$ there is a compact subset $K$ of   $X$ such that $K\cap U_n\ne\varnothing$ for each $n$ \cite{Gutik-Ravsky-20??};
  \item \emph{selectively sequentially feebly compact} if for every family $\{U_n\colon n\in \mathbb{N}\}$ of nonempty open subsets of $X$, one can choose a point $x_n\in U_n$ for every $n\in \mathbb{N}$ in such a way that the sequence $\{x_n\colon n\in \mathbb{N}\}$ has a convergent subsequence (\cite{Dorantes-Aldama-Shakhmatov-2017});
  \item \emph{sequentially feebly compact} if for every family $\{U_n\colon n\in \mathbb{N}\}$ of nonempty open subsets of $X$, there exists an infinite set $J\subseteq \mathbb{N}$ and a point $x\in X$ such that the set $\{n\in J\colon W\cap U_n=\varnothing\}$ is finite for every open neighborhood $W$ of $x$ (see \cite{Dow-Porter-Stephenson-Woods-2004});
  \item \emph{selectively feebly compact} for each sequence $\{U_n\colon n\in \mathbb{N}\}$ of nonempty open subsets of $X$, one can choose a point $x\in X$ and a point $x_n\in U_n$ for each $n\in \mathbb{N}$ such that the set $\{n\in \mathbb{N}\colon x_n\in W\}$ is infinite for every open neighborhood $W$ of $x$ (\cite{Dorantes-Aldama-Shakhmatov-2017});
  \item \emph{feebly compact} (or \emph{lightly compact}) if each locally finite open cover of $X$ is finite~\cite{Bagley-Connell-McKnight-Jr-1958};
  \item $d$-\emph{feebly compact} (or \emph{\textsf{DFCC}}) if every discrete family of open subsets in $X$ is finite (see \cite{Matveev-1998});
  \item \emph{pseudocompact} if $X$ is Tychonoff and each continuous real-valued function on $X$ is bounded;
  \item $Y$-\emph{compact} for some topological space $Y$, if $f(X)$ is compact, for any continuous map $f\colon X\to Y$.
\end{itemize}

According to Theorem~3.10.22 of \cite{Engelking-1989}, a Tychonoff topological space $X$ is feebly compact if and only if $X$ is pseudocompact. Also, a Hausdorff topological space $X$ is feebly compact if and only if every locally finite family of nonempty open subsets of $X$ is finite.  Every compact space and every sequentially compact space are countably compact, every countably compact space is countably pracompact, every countably pracompact space is feebly compact (see \cite{Arkhangelskii-1992}), every H-closed space is feebly compact too (see \cite{Gutik-Ravsky-2015a}). Also, every space feebly compact is infra H-closed by Proposition 2 and Theorem 3 of \cite{Hajek-Todd-1975}. Using results of other authors we get that the following diagram which describes relations between the above defined classes of topological spaces.

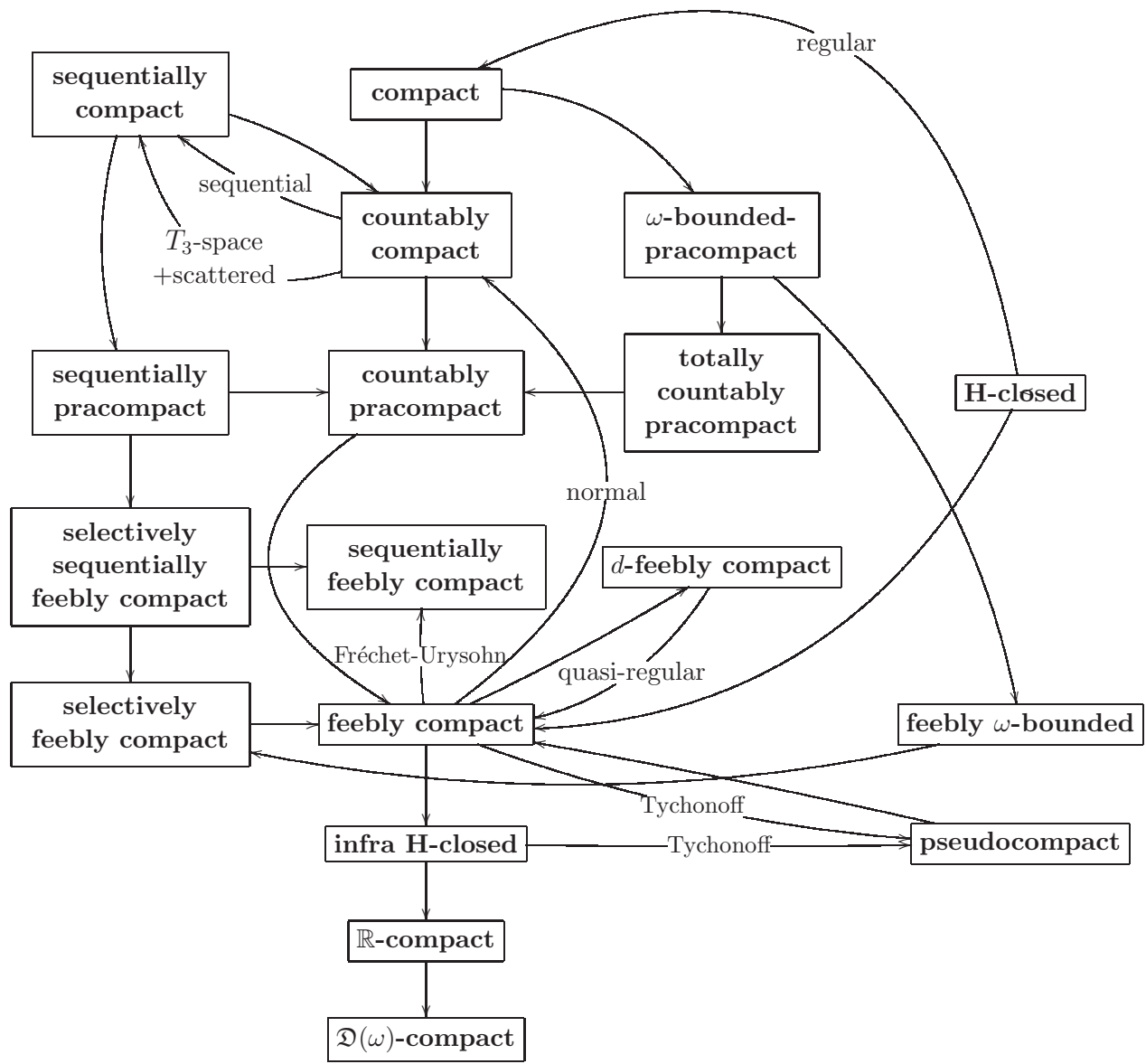
\begin{equation*}
\xymatrix{
*+[F]{\begin{array}{c}
                                                          \hbox{\textbf{\small{sequentially}}}\\
                                                          \hbox{\textbf{\small{compact}}}
                                                        \end{array}}\ar@/_1pc/[dd]\ar@/^1pc/[rd] &
*+[F]{\begin{array}{c}
                                                          \hbox{\textbf{\small{compact}}}
                                                        \end{array}
}\ar[d]\ar@/^2pc/[dr] & \\
 & *+[F]{\begin{array}{c}
                                                          \hbox{\textbf{\small{countably}}}\\
                                                          \hbox{\textbf{\small{compact}}}
                                                        \end{array}}\ar[d]
                                                        \ar@/^4pc/[lu]|-{\begin{array}{c}
                                                          \small{T_3\hbox{-space}}\\
                                                          \small{\hbox{+scattered}}
                                                        \end{array}} \ar@/^1pc/[lu]|-{\hbox{\small{sequential}}} &
 *+[F]{\begin{array}{c}
                                                          \hbox{\textbf{\small{$\omega$-bounded-}}}\\
                                                          \hbox{\textbf{\small{pracompact}}}
                                                        \end{array}}\ar[d]\ar@/^2pc/[rddd]\\
*+[F]{\begin{array}{c}
                                                          \hbox{\textbf{\small{sequentially}}}\\
                                                          \hbox{\textbf{\small{pracompact}}}
                                                        \end{array}}\ar[d]\ar[r]
 &*+[F]{\begin{array}{c}
                                                          \hbox{\textbf{\small{countably}}}\\
                                                          \hbox{\textbf{\small{pracompact}}}
                                                        \end{array}}\ar@/_5.5pc/[dd] &
*+[F]{\begin{array}{c}
                                                          \hbox{\textbf{\small{totally}}}\\
                                                          \hbox{\textbf{\small{countably}}}\\
                                                          \hbox{\textbf{\small{pracompact}}}
                                                        \end{array}}\ar[l]           &
*+[F]{\hbox{\textbf{\small{H-closed}}}}\ar\ar@/_-4.2pc/[lldd]\ar@/_8pc/[uull]|-{\small{\hbox{regular}}}                                              \\
*+[F]{\begin{array}{c}
                                                          \hbox{\textbf{\small{selectively}}}\\
                                                          \hbox{\textbf{\small{sequentially}}}\\
                                                          \hbox{\textbf{\small{feebly compact}}}
                                                        \end{array}}\ar[r]\ar[d]&
*+[F]{\begin{array}{c}
                                                          \hbox{\textbf{\small{sequentially}}}\\
                                                          \hbox{\textbf{\small{feebly compact}}}
                                                        \end{array}} & *+[F]{\hbox{\textbf{\small{$d$-feebly compact}}}} \ar@/_-1.9pc/[ld]|-{\small{\hbox{quasi-regular}}} &
                                                        \\
*+[F]{\begin{array}{c}
                                                          \hbox{\textbf{\small{selectively}}}\\
                                                          \hbox{\textbf{\small{feebly compact}}}
                                                        \end{array}}\ar[r]                                                          &
*+[F]{\hbox{\textbf{\small{feebly compact}}}}\ar@/_6.3pc/[uuu]|-{\small{\hbox{normal}}} \ar@/_-.2pc/[u]|-{\hbox{\footnotesize{Fr\'{e}chet-Urysohn}}}\ar@/_.2pc/[ur]\ar[d] \ar@/_1pc/[rrd]|-{\footnotesize{\hbox{Tychonoff}}}  & &
                                                        *+[F]{\hbox{\textbf{\small{feebly $\omega$-bounded}}}}\ar@/_-2.1pc/[lll]\\
&*+[F]{\hbox{\textbf{\small{infra H-closed}}}}\ar@/_.1pc/[rr]|-{\footnotesize{\hbox{Tychonoff}}}\ar[d]& & *+[F]{\hbox{\textbf{\small{pseudocompact}}}}\ar@/_.2pc/[ull]\\
&*+[F]{\hbox{\textbf{\small{$\mathbb{R}$-compact}}}}\ar[d] &&\\
&*+[F]{\hbox{\textbf{\small{$\mathfrak{D}(\omega)$-compact}}}} &&\\
&&&
}
\end{equation*}

A {\it topological} ({\it semitopological}) {\it semigroup} is a topological space together with a continuous (separately continuous) semigroup operation. If $S$ is a~semigroup and $\tau$ is a topology on $S$ such that  $(S,\tau)$ is a semitopological semigroup, then we shall call $\tau$ a \emph{shift-continuous} \emph{topology} on~$S$. An inverse topological semigroup with the continuous inversion is called a \emph{topological inverse semigroup}.

Topological properties of an infinite (semi)topological semigroup $\lambda\times \lambda$-matrix units were studied in \cite{Gutik-Pavlyk-2005a, Gutik-Pavlyk-Reiter-2009}. In \cite{Gutik-Pavlyk-2005a} it was shown that on the infinite semitopological semigroup of  $\lambda\times \lambda$-matrix units $B_\lambda$ there exists a unique compact shift-continuous Hausdorff topology $\tau_c$ and also it is shown that every pseudocompact Hausdorff shift-continuous topology $\tau$ on $B_\lambda$ is compact. Also, in \cite{Gutik-Pavlyk-2005a} it is proved that every nonzero element of a Hausdorff semitopological semigroup of $\lambda\times \lambda$-matrix units $B_\lambda$ is an isolated point in the topological space $B_\lambda$. In \cite{Gutik-Pavlyk-2005a} it is shown that the infinite semigroup of $\lambda\times \lambda$-matrix units $B_\lambda$ cannot be embedded into a compact Hausdorff topological semigroup, every Hausdorff topological inverse semigroup $S$ that contains $B_\lambda$ as a subsemigroup, contains $B_\lambda$ as a closed subsemigroup, i.e., $B_\lambda$ is \emph{algebraically complete} in the class of Hausdorff topological inverse semigroups. This result in \cite{Gutik-Lawson-Repov-2009} is extended onto so called inverse semigroups with \emph{tight ideal series} and, as a corollary, onto the semigroup $\mathscr{I}_\lambda^n$. Also, in \cite{Gutik-Reiter-2009} it was proved that for every positive integer $n$ the semigroup $\mathscr{I}_\lambda^n$ is \emph{algebraically $h$-complete} in the class of Hausdorff topological inverse semigroups, i.e., every homomorphic image of $\mathscr{I}_\lambda^n$ is algebraically complete in the class of Hausdorff topological inverse semigroups. In the paper \cite{Gutik-Reiter-2010} this result is extended onto the class of Hausdorff semitopological inverse semigroups and it is shown therein that for an infinite cardinal $\lambda$ the semigroup $\mathscr{I}_\lambda^n$ admits a unique Hausdorff topology $\tau_c$ such that $(\mathscr{I}_\lambda^n,\tau_c)$ is a compact semitopological semigroup. Also, it was proved in \cite{Gutik-Reiter-2010} that every countably compact Hausdorff shift-continuous topology $\tau$ on $B_\lambda$ is compact. In \cite{Gutik-Pavlyk-Reiter-2009} it was shown that a topological semigroup of finite partial bijections $\mathscr{I}_\lambda^n$ with a compact subsemigroup of idempotents is absolutely H-closed (i.e., every homomorphic image of $\mathscr{I}_\lambda^n$ is algebraically complete in the class of Hausdorff topological semigroups) and any Hausdorff 
countably compact topological semigroup does not contain $\mathscr{I}_\lambda^n$ as a subsemigroup for an arbitrary infinite cardinal $\lambda$ and any positive integer $n$. In \cite{Gutik-Pavlyk-Reiter-2009} there were given sufficient conditions onto a topological semigroup $\mathscr{I}_\lambda^1$ to be non-H-closed. Also in \cite{Gutik-2014} it is proved that an infinite semitopological semigroup of $\lambda\times\lambda$-matrix units $B_\lambda$ is H-closed in the class of semitopological semigroups if and only if the space $B_\lambda$ is compact.
In the paper \cite{Gutik-2017} we studied feebly compact shift-continuous $T_1$-topologies on the semigroup $\mathscr{I}_\lambda^n$. For any positive integer $n\geqslant2$ and any infinite cardinal $\lambda$ a Hausdorff countably pracompact non-compact shift-continuous topology on $\mathscr{I}_\lambda^n$ is constructed there. In \cite{Gutik-2017} it is shown that for an arbitrary positive integer $n$ and an arbitrary infinite cardinal $\lambda$ for a shift--continuous $T_1$-topology $\tau$ on $\mathscr{I}_\lambda^n$ the following conditions are equivalent: $(i)$ $\tau$ is countably pracompact; $(ii)$ $\tau$ is feebly compact; $(iii)$ $\tau$ is $d$-feebly compact; $(iv)$ $\left(\mathscr{I}_\lambda^n,\tau\right)$ is H-closed; $(v)$ $\left(\mathscr{I}_\lambda^n,\tau\right)$ is $\mathfrak{D}(\omega)$-compact;
$(vi)$ $\left(\mathscr{I}_\lambda^n,\tau\right)$ is $\mathbb{R}$-compact; $(vii)$ $\left(\mathscr{I}_\lambda^n,\tau\right)$ is  infra H-closed. Also in \cite{Gutik-2017} we proved that for an arbitrary positive integer $n$ and an arbitrary infinite cardinal $\lambda$  every shift-continuous semiregular feebly compact $T_1$-topology $\tau$ on $\mathscr{I}_\lambda^n$ is compact. Similar results were obtained for a semitopological semilattice $(\mathrm{exp}_n{\lambda},\cap)$ in \cite{Gutik-Sobol-2016, Gutik-Sobol-2016a, Gutik-Sobol-2018}. Also, in \cite{Gutik-Sobol-2021??, Lysetska-2020} it is proved that feeble compactness implies compactness for semitopological bicyclic extensions.

In this paper we study feebly compact shift-continuous $T_1$-topologies on the symmetric inverse semigroup $\mathscr{I}_\lambda^n$ of finite transformations of the rank $\leqslant n$. It is proved that such $T_1$-topology is sequentially pracompact if and only if it is feebly compact. Also, we show that every shift-continuous feebly $\omega$-bounded $T_1$-topology on $\mathscr{I}_\lambda^n$ is compact. The results of this paper is announced  in \cite{Gutik-2018}.

%%%%%%%%%%%%%%%%%%%%%%%%%%%%%%%%%%%%%%%%%%%%%%%%%%%%%%%%%%%%

\section{On feebly compact shift continuous topologies on the semigroup $\mathscr{I}_\lambda^n$}\label{section-2}

%\medskip

Later we shall assume that $n$ is an arbitrary positive integer.

For every element $\alpha$ of the semigroup $\mathscr{I}_\lambda^n$ we put
\begin{equation*}
  {\uparrow}_l\alpha=\left\{\beta\in\mathscr{I}_\lambda^n\colon \alpha\alpha^{-1}\beta=\alpha\right\} \qquad \hbox{and} \qquad
  {\uparrow}_r\alpha=\left\{\beta\in\mathscr{I}_\lambda^n\colon\beta\alpha^{-1}\alpha=\alpha\right\}.
\end{equation*}
Then  Proposition~5 of \cite{Gutik-Reiter-2010} implies that ${\uparrow}_l\alpha={\uparrow}_r\alpha$ and by Lemma~6 of \cite[Section~1.4]{Lawson-1998} we have that $\alpha\preccurlyeq\beta$ if and only if $\beta\in{\uparrow}_l\alpha$ for $\alpha,\beta\in\mathscr{I}_\lambda^n$. Hence we put ${\uparrow}_{\preccurlyeq}\alpha={\uparrow}_l\alpha={\uparrow}_r\alpha$ for any $\alpha\in\mathscr{I}_\lambda^n$.

\begin{remark}\label{remark-1}
Later we identify  every element $\alpha$ of the semigroup $\mathscr{I}_\lambda^n$ with the graph $\textsf{graph}(\alpha)$ of the partial map $\alpha\colon \lambda\rightharpoonup \lambda$ (see \cite{Lawson-1998}). Then according to this identification we have that  $\alpha\preccurlyeq\beta$ if and only if $\alpha\subseteq\beta$.
\end{remark}

\begin{lemma}\label{lemma-2}
Let $n$ be an arbitrary positive integer and $\lambda$ be any infinite cardinal. Let $\alpha$ be any nonzero element of the semigroup $\mathscr{I}_\lambda^n$ with $\operatorname{rank}\alpha=m\leqslant n$. Then the poset $({\uparrow}_{\preccurlyeq}\alpha,\preccurlyeq)$  is order isomorphic to the poset $(\mathscr{I}_\lambda^{n-m},\preccurlyeq)$.
\end{lemma}

\begin{proof}
Suppose that
\begin{equation*}
  \alpha=
\left(
  \begin{array}{ccc}
    x_1 & \cdots & x_m \\
    y_1 & \cdots & y_m \\
  \end{array}
\right)
\end{equation*}
for some $x_1,\ldots,x_m,y_1,\ldots,y_m\in\lambda$. If $m=n$ then the inequality $\alpha\preccurlyeq\beta$ in $(\mathscr{I}_\lambda^{n},\preccurlyeq)$ implies $\alpha=\beta$, and hence later we assume that $m<n$. Then for any $\beta\in\mathscr{I}_\lambda^{n}$ such that $\alpha\preccurlyeq\beta$ by Remark~\ref{remark-1} we have that
\begin{equation*}
  \beta=
\left(
  \begin{array}{cccccc}
    x_1 & \cdots & x_m & x_{m+1} & \cdots & x_n\\
    y_1 & \cdots & y_m & y_{m+1} & \cdots & y_n\\
  \end{array}
\right)
\end{equation*}
for some $x_{m+1},\ldots,x_n,y_{m+1},\ldots,y_n\in\lambda$. Since $\lambda$ is infinite, $|\lambda|=|\lambda\setminus\{x_1,\ldots,x_m\}|=|\lambda\setminus\{y_1,\ldots,y_m\}|$, and hence there exist bijective maps $\mathfrak{u}\colon \lambda\setminus\{x_1,\ldots,x_m\}\to \lambda$ and $\mathfrak{v}\colon \lambda\setminus\{y_1,\ldots,y_m\}\to\lambda$. Simple verifications show that the map $\mathfrak{I}\colon ({\uparrow}_{\preccurlyeq}\alpha,\preccurlyeq)\to (\mathscr{I}_\lambda^{n-m},\preccurlyeq)$ defined in the following way $\alpha\mapsto\boldsymbol{0}$ and
\begin{equation*}
\left(
  \begin{array}{cccccc}
    x_1 & \cdots & x_m & x_{m+1} & \cdots & x_n\\
    y_1 & \cdots & y_m & y_{m+1} & \cdots & y_n\\
  \end{array}
\right)
\mapsto
\left(
  \begin{array}{cccccc}
    (x_{m+1})\mathfrak{u} & \cdots & (x_n)\mathfrak{u}\\
    (y_{m+1})\mathfrak{v} & \cdots & (y_n)\mathfrak{v}\\
  \end{array}
\right)
\end{equation*}
is an order isomorphism.
\end{proof}

Later we need the following technical lemma from \cite{Gutik-2017}.

\begin{lemma}[{\cite[Lemma~3]{Gutik-2017}}]\label{lemma-2.10}
Let $n$ be an arbitrary positive integer and $\lambda$ be an arbitrary infinite cardinal. Let $\tau$ be a feebly compact shift-continuous $T_1$-topology on the semigroup $\mathscr{I}_\lambda^n$. Then for every $\alpha\in\mathscr{I}_\lambda^n$ and any open neighbourhood $U(\alpha)$ of $\alpha$ in $\left(\mathscr{I}_\lambda^n,\tau\right)$ there exist finitely many $\alpha_1,\ldots,\alpha_k\in{\uparrow}_{\preccurlyeq}\alpha\setminus\{\alpha\}$ such that
\begin{equation*}
  \mathscr{I}_\lambda^n\setminus\mathscr{I}_\lambda^{n-1}\cap {\uparrow}_{\preccurlyeq}\alpha \subseteq U(\alpha)\cup{\uparrow}_{\preccurlyeq}\alpha_1\cup \cdots\cup {\uparrow}_{\preccurlyeq}\alpha_k.
\end{equation*}
\end{lemma}

\begin{lemma}\label{lemma-3}
Let $\tau$ be a feebly compact topology on $\mathscr{I}_\lambda^1$ such that ${\uparrow}_{\preccurlyeq}\alpha$ is closed-and-open for any $\alpha\in \mathscr{I}_\lambda^1$. Then $\tau$ is compact.
\end{lemma}

The statement of Lemma~\ref{lemma-3} follows from the fact that all nonzero elements of the semigroup $\mathscr{I}_\lambda^1$ are closed-and-open in $(\mathscr{I}_\lambda^1,\tau)$.

A family of non-empty sets $\{A_i\colon i\in \mathscr{I}\}$  is called a \emph{$\Delta$-system} (a \emph{sunflower} or a \emph{$\Delta$-family}) if the pairwise intersections of the members are the same, i.e., $A_i\cap A_j=S$ for some set $S$ (for $i\neq j$ in $\mathscr{I}$) \cite{Komjath-Totik-2006}.
The following statement is well known as the \emph{Sunflower Lemma} or the \emph{Lemma about a $\Delta$-system} (see \cite[p. 107]{Komjath-Totik-2006}).

\begin{lemma}\label{Subflower_Lemma}
Every infinite family of $n$-element sets $(n<\omega)$ contains an infinite $\Delta$-subfamily.
\end{lemma}

%\newpage

\begin{proposition}\label{proposition-6}
Let $n$ be an arbitrary positive integer and $\lambda$ be an arbitrary infinite cardinal. Then every
feebly compact shift-continuous $T_1$-topology $\tau$ on $\mathscr{I}_\lambda^n$ is sequentially pracompact.
\end{proposition}

\begin{proof}
Suppose to the contrary that there exists a feebly compact shift-continuous $T_1$-topology $\tau$ on $\mathscr{I}_\lambda^n$ which is not sequentially countably pracompact. Then every dense subset $D$ of $(\mathscr{I}_\lambda^n,\tau)$ contains a sequence of points from $D$ which has no  a convergent subsequence.

By Proposition~2 of \cite{Gutik-2017} the subset $\mathscr{I}_\lambda^n\setminus\mathscr{I}_\lambda^{n-1}$ is dense in $\left(\mathscr{I}_\lambda^n,\tau\right)$ and by Lemma~2 from \cite{Gutik-2017}  every point of the set $\mathscr{I}_\lambda^n\setminus\mathscr{I}_\lambda^{n-1}$ is isolated in $\left(\mathscr{I}_\lambda^n,\tau\right)$. Then the set $\mathscr{I}_\lambda^n\setminus\mathscr{I}_\lambda^{n-1}$ contains an infinite sequence of points $\{\chi_p\colon p\in \mathbb{N}\}$ which has not a convergent subsequence. If we identify elements of the semigroups with their graphs then by Lemma~\ref{Subflower_Lemma} the sequence $\{\chi_p\colon p\in \mathbb{N}\}$ contains an infinite $\Delta$-subfamily, that is an infinite subsequence $\{\chi_{p_i}\colon i\in \mathbb{N}\}$ such that there exists $\chi\in \mathscr{I}_\lambda^n$ such that $\chi_{p_i}\cap \chi_{p_j}=\chi$  for any distinct $i,j\in\mathbb{N}$.

Suppose that $\chi=\boldsymbol{0}$ is the zero of the semigroup $\mathscr{I}_\lambda^n$. Since the sequence $\{\chi_{p_i}\colon i\in \mathbb{N}\}$ is an infinite $\Delta$-subfamily, the intersection $\{\chi_{p_i}\colon i\in \mathbb{N}\}\cap{\uparrow}_{\preccurlyeq}\gamma$ contains at most one set for every non-zero element $\gamma\in\mathscr{I}_\lambda^n$. Thus $(\mathscr{I}_\lambda^n,\tau)$ contains an infinite locally finite family of open non-empty subsets which contradicts the feeble compactness of $(\mathscr{I}_\lambda^n,\tau)$.

If $\chi$ is a non-zero element of the  semigroup $\mathscr{I}_\lambda^n$ then by Lemma~2 from \cite{Gutik-2017}, ${\uparrow}_{\preccurlyeq}\chi$ is an open-and-closed subspace of $(\mathscr{I}_\lambda^n,\tau)$, and hence by Theorem~14 from \cite{Bagley-Connell-McKnight-Jr-1958} the space ${\uparrow}_{\preccurlyeq}\chi$ is feebly compact. We observe that the element $\chi$ is the minimum of the poset ${\uparrow}_{\preccurlyeq}\chi$. Since the sequence $\{\chi_{p_i}\colon i\in \mathbb{N}\}$ is an infinite $\Delta$-subfamily, the intersection $\{\chi_{p_i}\colon i\in \mathbb{N}\}\cap{\uparrow}_{\preccurlyeq}\gamma$ contains at most one set for every element $\gamma\in{\uparrow}_{\preccurlyeq}\chi\setminus\{\chi\}$. Thus the subspace ${\uparrow}_{\preccurlyeq}\chi$ of $(\mathscr{I}_\lambda^n,\tau)$ contains an infinite locally finite family of open non-empty subsets which contradicts the feeble compactness of $(\mathscr{I}_\lambda^n,\tau)$.
\end{proof}

\begin{proposition}\label{proposition-7}
Let $n$ be an arbitrary positive integer and $\lambda$ be an arbitrary infinite cardinal. Then every feebly compact shift-continuous $T_1$-topology $\tau$ on $\mathscr{I}_\lambda^n$ is totally countably pracompact.
\end{proposition}

\begin{proof}
By Proposition~2 of \cite{Gutik-2017} the subset $\mathscr{I}_\lambda^n\setminus\mathscr{I}_\lambda^{n-1}$ is dense in $\left(\mathscr{I}_\lambda^n,\tau\right)$ and by Lemma~2 from \cite{Gutik-2017}  every point of the set $\mathscr{I}_\lambda^n\setminus\mathscr{I}_\lambda^{n-1}$ is isolated in $\left(\mathscr{I}_\lambda^n,\tau\right)$. We put $D=\mathscr{I}_\lambda^n\setminus\mathscr{I}_\lambda^{n-1}$. Fix an arbitrary sequence $\{\chi_p\colon p\in \mathbb{N}\}$ of points of $D$.

It is obvious that at least one of the following conditions holds:
\begin{enumerate}
  \item[\textbf{(1)}] for any $\eta\in \mathscr{I}_\lambda^n\setminus\{\boldsymbol{0}\}$ the set ${\uparrow}_{\preccurlyeq}\eta \cap \{\chi_p\colon p\in \mathbb{N}\}$ is finite;
  \item[\textbf{(2)}] there exists $\eta\in \mathscr{I}_\lambda^n\setminus\{\boldsymbol{0}\}$ such that the set ${\uparrow}_{\preccurlyeq}\eta \cap \{\chi_p\colon p\in \mathbb{N}\}$ is infinite.
\end{enumerate}

Suppose case \textbf{(1)} holds. By Lemma~2 of \cite{Gutik-2017} for every point $\alpha\in \mathscr{I}_\lambda^n\setminus\{0\}$ there exists an open neighbourhood $U(\alpha)$ of $\alpha$ in $\left(\mathscr{I}_\lambda^n,\tau\right)$ such that $U(\alpha)\subseteq {\uparrow}_{\preccurlyeq}\alpha$ and hence our assumption implies that zero $\boldsymbol{0}$ is a unique accumulation point of the sequence $\{\chi_p\colon p\in \mathbb{N}\}$. By Lemma~\ref{lemma-2.10}  for an arbitrary open neighbourhood $W(\boldsymbol{0})$ of zero $\boldsymbol{0}$ in $\left(\mathscr{I}_\lambda^n,\tau\right)$ there exist finitely many nonzero elements $\eta_1,\ldots, \eta_{k}\in \mathscr{I}_\lambda^n$ such that
\begin{equation*}
  \left(\mathscr{I}_\lambda^n\setminus\mathscr{I}_\lambda^{n-1}\right)\subseteq W(\boldsymbol{0})\cup {\uparrow}_{\preccurlyeq}{\eta_1}\cup\cdots\cup {\uparrow}_{\preccurlyeq}{\eta_{k}},
\end{equation*}
and hence we get that $\{\boldsymbol{0}\}\cup\{\chi_p\colon p\in \mathbb{N}\}$ is a compact subset of $\left(\mathscr{I}_\lambda^n,\tau\right)$.

Suppose case \textbf{(2)} holds: there exists $\eta^1\in \mathscr{I}_\lambda^n\setminus\{\boldsymbol{0}\}$ such that the set ${\uparrow}_{\preccurlyeq}\eta^1 \cap \{\chi_p\colon p\in \mathbb{N}\}$ is infinite. Then by Lemma~2 of \cite{Gutik-2017}, ${\uparrow}_{\preccurlyeq}y^1$ is an open-and-closed subset of $\left(\mathscr{I}_\lambda^n,\tau\right)$ and hence by Theorem~14 from \cite{Bagley-Connell-McKnight-Jr-1958} the subspace ${\uparrow}_{\preccurlyeq}\eta^{1}$ of $\left(\mathscr{I}_\lambda^n,\tau\right)$ is feebly compact. By Lemma~\ref{lemma-2} the poset $({\uparrow}_{\preccurlyeq}\eta^1,\preccurlyeq)$ is order isomorphic to the poset $(\mathscr{I}_\lambda^{m_1},\preccurlyeq)$ for some positive integer $m_{1}=2,\ldots,n-1$.

We put $\{\chi^1_p\colon p\in \mathbb{N}\}$ is a subsequence of $\{\chi_p\colon p\in \mathbb{N}\}$ such that $\{\chi^1_p\colon p\in \mathbb{N}\}={\uparrow}_{\preccurlyeq}\eta^1 \cap \{\chi_p\colon p\in \mathbb{N}\}$. Then for the feebly compact poset $({\uparrow}_{\preccurlyeq}\eta^1,\preccurlyeq)$ and the sequence $\{\chi^1_p\colon p\in \mathbb{N}\}$ at least one of the following conditions holds:
\begin{enumerate}
  \item[\textbf{(1)$_*$}] for any $\eta\in {\uparrow}_{\preccurlyeq}\eta^1\setminus\{\eta^1\}$ the set ${\uparrow}_{\preccurlyeq}\eta \cap \{\chi_p^1\colon p\in \mathbb{N}\}$ is finite;
  \item[\textbf{(2)$_*$}] there exists $\eta\in {\uparrow}_{\preccurlyeq}\eta^1\setminus\{\eta^1\}$ such that the set ${\uparrow}_{\preccurlyeq}\eta \cap \{\chi_p^1\colon p\in \mathbb{N}\}$ is infinite.
\end{enumerate}

Since every chain in the poset $({\uparrow}_{\preccurlyeq}\eta^1,\preccurlyeq)$ is finite, repeating finitely many times our above procedure we obtain two chains of the length $s\leqslant n$:
\begin{itemize}
  \item[$(i)$] the chain $\boldsymbol{0}\preccurlyeq\eta^1\preccurlyeq\cdots\preccurlyeq\eta^s$ of distinct  elements of the poset $({\uparrow}_{\preccurlyeq}\eta^1,\preccurlyeq)$; \qquad and
  \item[$(ii)$] the chain $\{\chi_p\colon p\in \mathbb{N}\}\supseteq \{\chi^1_p\colon p\in \mathbb{N}\}\supseteq\cdots\supseteq\{\chi^s_p\colon p\in \mathbb{N}\}$ of infinite subsequences of the sequence $\{\chi_p\colon p\in \mathbb{N}\}$,
\end{itemize}
such that the following conditions hold:
\begin{itemize}
  \item[$(a)$] $\{\chi^j_p\colon p\in \mathbb{N}\}\subseteq{\uparrow}_{\preccurlyeq}\eta^j$ for every $j=1,\ldots,s$;
  \item[$(b)$] either $\{\chi^s_p\colon p\in \mathbb{N}\}\cup\{\eta^s\}$ is a compact subset of the poset $({\uparrow}_{\preccurlyeq}\eta^1,\preccurlyeq)$ or the poset $({\uparrow}_{\preccurlyeq}\eta^s,\preccurlyeq)$ is order isomorphic to the poset $(\mathscr{I}_\lambda^{1},\preccurlyeq)$.
\end{itemize}

If $\{\chi^s_p\colon p\in \mathbb{N}\}\cup\{\eta^s\}$ is a compact subset of $\left(\mathscr{I}_\lambda^n,\tau\right)$ then our above part of the proof implies that the sequence $\{\chi_p\colon p\in \mathbb{N}\}$ has the subsequence $\{\chi^s_p\colon p\in \mathbb{N}\}$ with the compact closure.

If the poset $({\uparrow}_{\preccurlyeq}\eta^s,\preccurlyeq)$ is order isomorphic to the poset $(\mathscr{I}_\lambda^{1},\preccurlyeq)$, then by Lemma~2 of \cite{Gutik-2017} the subspace ${\uparrow}_{\preccurlyeq}\eta^s$ of $\left(\mathscr{I}_\lambda^n,\tau\right)$ is open-and-closed and hence by Lemmas~\ref{lemma-2} and~\ref{lemma-3}  the poset $({\uparrow}_{\preccurlyeq}\eta^s,\preccurlyeq)$ is compact. Then the inclusion $\{\chi^s_p\colon p\in \mathbb{N}\}\subseteq{\uparrow}_{\preccurlyeq}\eta^s$ implies that the sequence $\{\chi_p\colon p\in \mathbb{N}\}$ has the subsequence $\{\chi^s_p\colon p\in \mathbb{N}\}$ with the compact closure. This completed the proof of the proposition.
\end{proof}

We summarise our results in the  following theorem.

\begin{theorem}\label{theorem-8}
Let $n$ be any positive integer and $\lambda$ be any infinite cardinal. Then for any $T_1$-semitopological semigroup $\mathscr{I}_\lambda^{n}$ the following conditions are equivalent:
\begin{itemize}
  \item[$(i)$] $\mathscr{I}_\lambda^{n}$ is sequentially pracompact;
  \item[$(ii)$] $\mathscr{I}_\lambda^{n}$ is totally countably pracompact;
  \item[$(iii)$] $\mathscr{I}_\lambda^{n}$ is feebly compact.
\end{itemize}
\end{theorem}

\begin{proof}
Implications $(i)\Rightarrow(iii)$ and $(ii)\Rightarrow(iii)$ are trivial. The corresponding their converse implications $(iii)\Rightarrow(i)$ and $(iii)\Rightarrow(ii)$ follow from Propositions~\ref{proposition-6} and \ref{proposition-7}, respectively.
\end{proof}

It is well known that the (Tychonoff) product of pseudocompact spaces is not necessarily pseudocompact (see \cite[Section~3.10]{Engelking-1989}). On the other hand Comfort and Ross in \cite{Comfort-Ross-1966} proved that a Tychonoff product of an arbitrary family of pseudocompact topological groups is a pseudocompact topological group.  The Comfort--Ross Theorem is generalized in \cite{Banakh-Ravsky-2020} and it is proved that a Tychonoff product of an arbitrary non-empty family of feebly compact paratopological groups is feebly compact. Also, a counterpart of the Comfort--Ross Theorem for pseudocompact primitive topological inverse semigroups and  primitive inverse semiregular feebly compact semitopological semigroups with closed maximal subgroups were proved in \cite{Gutik-Pavlyk-2013} and \cite{Gutik-Ravsky-2015}, respectively.

Since a Tychonoff product of H-closed spaces is H-closed (see \cite[Theorem~3]{Chevalley-Frink-1941} or \cite[3.12.5~(d)]{Engelking-1989}) Theorem~\ref{theorem-8} implies a counterpart of the Comfort--Ross Theorem for feebly compact semitopological semigroups $\mathscr{I}_\lambda^{n}$:

\begin{corollary}\label{corollary-9}
Let $\left\{\mathscr{I}_{\lambda_i}^{n_i}\colon i\in\mathscr{J}\right\}$ be a family of non-empty feebly compact $T_1$-semitopological semigroups and $n_i\in\mathbb{N}$ for all $i\in\mathscr{I}$. Then the Tychonoff product $\prod\left\{\mathscr{I}_{\lambda_i}^{n_i}\colon i\in\mathscr{I}\right\}$ is feebly compact.
\end{corollary}

\begin{definition}\label{definition-10}
If $\left\{X_i\colon i\in \mathscr{I}\right\}$ is an uncountable family of sets, $X=\prod\left\{X_i\colon i\in \mathscr{J}\right\}$ is their Cartesian product and $p$ is a point in $X$,
then the subset
\begin{equation*}
  \Sigma(p,X)=\left\{ x \in X\colon \left| \left\{ i\in\mathscr{J}\colon x(i)\neq p(i)\right\}\right|\leqslant\omega\right\}
\end{equation*}
of $X$ is called the \emph{$\Sigma$-product} of $\left\{X_i\colon i\in \mathscr{J}\right\}$ with the basis point $p\in X$. In the case when $\left\{X_i\colon i\in \mathscr{J}\right\}$ is a family of topological spaces we assume that $\Sigma(p,X)$ is a subspace of the Tychonoff product $X=\prod\left\{X_i\colon i\in \mathscr{J}\right\}$.
\end{definition}

It is obvious that if $\left\{X_i\colon i\in \mathscr{J}\right\}$ is a family of semigroups then $X=\prod\left\{X_i\colon i\in \mathscr{J}\right\}$  is a semigroup as well. Moreover $\Sigma(p,X)$ is a subsemigroup of $X$ for arbitrary idempotent $p\in X$. Theorem~\ref{theorem-8} and Proposition~2.2 of \cite{Gutik-Ravsky-20??} imply the following corollary.

\begin{corollary}\label{corollary-13}
Let $\left\{\mathscr{I}_{\lambda_i}^{n_i}\colon i\in\mathscr{J}\right\}$  be a family of non-empty feebly compact $T_1$-semitopologi\-cal semigroups and $n_i\in\mathbb{N}$ for all $i\in\mathscr{J}$. Then for every idempotent $p$ of the product $X=\prod\left\{\mathscr{I}_{\lambda_i}^{n_i}\colon i\in\mathscr{J}\right\}$ the $\Sigma$-product $\Sigma(p,X)$ is feebly compact.
\end{corollary}

\section{On compact shift continuous topologies on the semigroup $\mathscr{I}_\lambda^n$}\label{section-2}

The following example implies that there exists a countable feebly compact Hausdorff semitopological semigroup $\left(\mathscr{I}_\omega^{2},\right)$ which is not $\omega$-bounded-pracompact.

\begin{example}\label{example-14}
The following family
\begin{equation*}
\begin{split}
  \mathscr{B}_{\operatorname{\textsf{c}}}& =\left\{U_\alpha(\alpha_1,\ldots,\alpha_k)=
    {\uparrow}_{\preccurlyeq}\alpha\setminus({\uparrow}_{\preccurlyeq}\alpha_1\cup\cdots\cup
    {\uparrow}_{\preccurlyeq}\alpha_k) \colon\right. \\
    & \qquad \left. \alpha_i\in{\uparrow}_{\preccurlyeq}\alpha\setminus
    \{\alpha\}, \alpha, \alpha_i\in\mathscr{I}_\omega^2, i=1,\ldots, k\right\}
\end{split}
\end{equation*}
determines a base of the topology $\tau_{\operatorname{\textsf{c}}}$ on $\mathscr{I}_\omega^2$. By Proposition~10 from \cite{Gutik-Reiter-2010}, $\left(\mathscr{I}_\omega^2,\tau_{\operatorname{\textsf{c}}}\right)$ is a Hausdorff compact semitopological semigroup with continuous inversion.

We construct a  stronger topology $\tau_{\operatorname{\textsf{fc}}}^2$ on $\mathscr{I}_\lambda^2$ in the following way. For every nonzero element $x\in\mathscr{I}_\lambda^2$ we assume that the base $\mathscr{B}_{\operatorname{\textsf{fc}}}^2(x)$ of the topology $\tau_{\operatorname{\textsf{fc}}}^2$ at the point $x$ coincides with the base of the topology $\tau_{\operatorname{\textsf{c}}}^2$ at $x$, and
\begin{equation*}
  \mathscr{B}_{\operatorname{\textsf{fc}}}^2(0)  =\left\{U_B(\operatorname{\textbf{0}})=U(\operatorname{\textbf{0}})\setminus\left(\mathscr{I}_\lambda^2\setminus\{\boldsymbol{0}\}\right\}\colon U(0)\in\mathscr{B}_{\operatorname{\textsf{c}}}^2(0)  \right\}
\end{equation*}
form a base of the topology $\tau_{\operatorname{\textsf{fc}}}^2$ at zero $\operatorname{\textbf{0}}$ of the semigroup $\mathscr{I}_\lambda^2$. Since $\left(\mathscr{I}_\omega^2,\tau_{\operatorname{\textsf{fc}}}^2\right)$ is a variant of the semitopological semigroup defined in Example 3 of \cite{Gutik-2017}, $\tau_{\operatorname{\textsf{fc}}}^2$ is a Hausdorff topology on $\mathscr{I}_\lambda^2$. Moreover, by Proposition~1 of \cite{Gutik-2017}, $\left(\mathscr{I}_\omega^2,\tau_{\operatorname{\textsf{fc}}}^2\right)$ is a countably pracompact semitopological semigroup with continuous inversion.
\end{example}

\begin{proposition}\label{proposition-15}
The space $\left(\mathscr{I}_\omega^2,\tau_{\operatorname{\textsf{fc}}}^2\right)$ is not $\omega$-bounded-pracompact.
\end{proposition}

\begin{proof}
Since the space $\left(\mathscr{I}_\omega^2,\tau_{\operatorname{\textsf{fc}}}^2\right)$ is feebly compact and Hausdorff, by Proposition~2 of \cite{Gutik-2017} the subset $\mathscr{I}_\lambda^2\setminus\mathscr{I}_\lambda^{1}$ is dense in $\left(\mathscr{I}_\omega^2,\tau_{\operatorname{\textsf{fc}}}^2\right)$, and by Lemma~2 from \cite{Gutik-2017}  every point of the set $\mathscr{I}_\lambda^2\setminus\mathscr{I}_\lambda^{1}$ is isolated in $\left(\mathscr{I}_\omega^2,\tau_{\operatorname{\textsf{fc}}}^2\right)$. This implies that every dense subset $D$ of $\left(\mathscr{I}_\omega^2,\tau_{\operatorname{\textsf{fc}}}^2\right)$ contains the set $\mathscr{I}_\lambda^n\setminus\mathscr{I}_\lambda^{n-1}$. Then
\begin{equation*}
\operatorname{cl}_{\left(\mathscr{I}_\omega^2,\tau_{\operatorname{\textsf{fc}}}^2\right)}(D)= \operatorname{cl}_{\left(\mathscr{I}_\omega^2,\tau_{\operatorname{\textsf{fc}}}^2\right)}(\mathscr{I}_\lambda^2\setminus\mathscr{I}_\lambda^{1})= \mathscr{I}_\omega^2
\end{equation*}
for every dense subset $D$ of $\left(\mathscr{I}_\omega^2,\tau_{\operatorname{\textsf{fc}}}^2\right)$. Since $\mathscr{I}_\omega^2$ is countable, so is $D$, and hence the space $\left(\mathscr{I}_\omega^2,\tau_{\operatorname{\textsf{fc}}}^2\right)$ is not $\omega$-bounded-pracompact, because $\left(\mathscr{I}_\omega^2,\tau_{\operatorname{\textsf{fc}}}^2\right)$ is not compact.
\end{proof}

\begin{proposition}\label{proposition-16}
Let $n$ be any positive integer and $\lambda$ be any infinite cardinal. If $\mathscr{I}_\lambda^{n}$ is a $T_1$-semitopological semigroup then the following statements hold:
\begin{enumerate}
  \item\label{proposition-16-1} $\mathscr{I}_A^{n}$ is a closed subsemigroup  of $\mathscr{I}_\lambda^{n}$ for any subset $A\subseteq \lambda$;
  \item\label{proposition-16-2} the band $E(\mathscr{I}_\lambda^{n})$ is a closed subset  of $\mathscr{I}_\lambda^{n}$.
\end{enumerate}
\end{proposition}

\begin{proof}
\eqref{proposition-16-1} Fix an arbitrary $\gamma\in \mathscr{I}_\lambda^{n}\setminus \mathscr{I}_A^{n}$. Then $\operatorname{dom}\gamma\nsubseteq A$ or $\operatorname{ran}\gamma\nsubseteq A$. Since $\eta\preccurlyeq\delta$ if and only if $\textsf{graph}(\eta)\subseteq \textsf{graph}(\delta)$ for $\eta,\delta\in \mathscr{I}_\lambda^{n}$, the above arguments imply that ${\uparrow}_{\preccurlyeq}\gamma\cap \mathscr{I}_A^{n}=\varnothing$. By Lemma~2 of \cite{Gutik-2017} the set ${\uparrow}_{\preccurlyeq}\gamma$ is open in $\mathscr{I}_\lambda^{n}$, which implies statement \eqref{proposition-16-1}.

\eqref{proposition-16-2} Fix an arbitrary $\gamma\in \mathscr{I}_\lambda^{n}\setminus E(\mathscr{I}_\lambda^{n})$. Since $\mathscr{I}_\lambda^{n}$ is an inverse subsemigroup of the symmetric inverse monoid  $\mathscr{I}_\lambda$, all idempotents of $\mathscr{I}_\lambda^{n}$ is are partial identity maps of rank $\leqslant n$. Then similar arguments as in statement \eqref{proposition-16-1} imply that $E(\mathscr{I}_\lambda^{n})$ is a closed subset  of $\mathscr{I}_\lambda^{n}$.
\end{proof}

Proposition~\ref{proposition-16} implies the following corollary.

\begin{corollary}\label{corollary-17}
Let $n$ be any positive integer, $\lambda$ be any infinite cardinal and $A$ be an arbitrary infinite subset of $\lambda$. If $\mathscr{I}_\lambda^{n}$ is a compact $T_1$-semitopological semigroup then $\mathscr{I}_A^{n}$ with the induced topology from $\mathscr{I}_\lambda^{n}$ is a compact semitopological semigroup.
\end{corollary}

\begin{lemma}\label{lemma-18}
Let $n$ be any positive integer, $\lambda$ be any infinite cardinal and $A$ be an arbitrary infinite countable subset of $\lambda$. If $\mathscr{I}_\lambda^{n}$ is a $\omega$-bounded-pracompact $T_1$-semi\-topolo\-gical semigroup then $\mathscr{I}_A^{n}\setminus\mathscr{I}_A^{n-1}$ is a dense subset of $\mathscr{I}_A^{n}$, and hence $\mathscr{I}_A^{n}$ is compact.
\end{lemma}

\begin{proof}
For any $\alpha\in \mathscr{I}_A^{n}$ we denote ${\uparrow}_{\preccurlyeq}^A\alpha={\uparrow}_{\preccurlyeq}\alpha\cap \mathscr{I}_A^{n}$.

By induction we shall show that the set ${\uparrow}_{\preccurlyeq}^A\alpha\cap (\mathscr{I}_A^{n}\setminus\mathscr{I}_A^{n-1})$ is dense in ${\uparrow}_{\preccurlyeq}^A\alpha$  for any $\alpha\in \mathscr{I}_A^{n}$. In the case when $\operatorname{rank}\alpha=n-1$ by Lemmas~\ref{lemma-2} and \ref{lemma-3} we have that the set ${\uparrow}_{\preccurlyeq}\alpha$ is compact, and hence by Proposition~\ref{proposition-16}\eqref{proposition-16-1}, ${\uparrow}_{\preccurlyeq}^A\alpha$ is compact as well. Since all points of $\mathscr{I}_A^{n}\setminus\mathscr{I}_A^{n-1}$ are isolated in $\mathscr{I}_\lambda^{n}$, the set ${\uparrow}_{\preccurlyeq}^A\alpha\cap (\mathscr{I}_A^{n}\setminus\mathscr{I}_A^{n-1})$ is dense in ${\uparrow}_{\preccurlyeq}^A\alpha$.

Next we show that the statement \emph{${\uparrow}_{\preccurlyeq}^A\alpha\cap (\mathscr{I}_A^{n}\setminus\mathscr{I}_A^{n-1})$ is dense in ${\uparrow}_{\preccurlyeq}^A\alpha$  for any $\alpha\in \mathscr{I}_A^{n}$ with $\operatorname{rank}\alpha=n-k$, for all $k<m$} implies that the same is true for any $\beta\in \mathscr{I}_A^{n}$ with $\operatorname{rank}\beta=n-m$, where $m\leqslant n$. Fix an arbitrary $\beta\in \mathscr{I}_A^{n}$ with $\operatorname{rank}\beta=n-m$. Suppose to the contrary that the set ${\uparrow}_{\preccurlyeq}^A\beta\cap (\mathscr{I}_A^{n}\setminus\mathscr{I}_A^{n-1})$ is not dense in ${\uparrow}_{\preccurlyeq}^A\beta$. The assumption of induction implies that $\gamma\in \operatorname{cl}_{\mathscr{I}_A^{n}}({\uparrow}_{\preccurlyeq}^A\beta\cap (\mathscr{I}_A^{n}\setminus\mathscr{I}_A^{n-1}))$ for any $\gamma\in {\uparrow}_{\preccurlyeq}^A\beta\setminus\{\beta\}$, and hence $\beta\notin\operatorname{cl}_{\mathscr{I}_A^{n}}({\uparrow}_{\preccurlyeq}^A\beta\cap (\mathscr{I}_A^{n}\setminus\mathscr{I}_A^{n-1}))$. Then there exists an open neighbourhood $U(\beta)$ of $\beta$ in $\mathscr{I}_A^{n}$ such that $U(\beta)\cap ({\uparrow}_{\preccurlyeq}^A\beta\cap (\mathscr{I}_A^{n}\setminus\mathscr{I}_A^{n-1}))=\varnothing$. By Lemma~2 from \cite{Gutik-2017} for any $\delta\in\mathscr{I}_\lambda^n$ the set ${\uparrow}_{\preccurlyeq}\delta$ is open-and-closed in $\mathscr{I}_\lambda^n,\tau$, and hence ${\uparrow}_{\preccurlyeq}^A\delta$ is open-and-closed in $\mathscr{I}_A^n$ as well. Hence we get that
\begin{equation*}
\operatorname{cl}_{\mathscr{I}_A^{n}}({\uparrow}_{\preccurlyeq}^A\beta\cap (\mathscr{I}_A^{n}\setminus\mathscr{I}_A^{n-1}))={\uparrow}_{\preccurlyeq}^A\beta\setminus\{\beta\}
\end{equation*}
but the family $\mathscr{U}=\left\{{\uparrow}_{\preccurlyeq}^A\delta\colon \delta\in {\uparrow}_{\preccurlyeq}^A\beta\setminus\{\beta\} \right\}$ is an open cover of ${\uparrow}_{\preccurlyeq}^A\beta$ which hasn't a finite subcover. This contradicts the condition that $\mathscr{I}_\lambda^{n}$ is a $\omega$-bounded-pracompact space, which completes the proof of the first statement of the lemma. The last statement immediately follows from the firs statement and the definition of the $\omega$-bounded-pracompact space.
\end{proof}

Theorem~\ref{theorem-19} describes feebly $\omega$-bounded shift-continuous $T_1$-topologies on the semigroup $\mathscr{I}_\omega^n$.

\begin{theorem}\label{theorem-19}
Let $n$ be any positive integer and $\lambda$ be any infinite cardinal. Then for any $T_1$-semi\-topolo\-gical semigroup $\mathscr{I}_\lambda^{n}$ the following conditions are equivalent:
\begin{itemize}
  \item[$(i)$] $\mathscr{I}_\lambda^{n}$ compact;
  \item[$(ii)$] $\mathscr{I}_\lambda^{n}$ is $\omega$-bounded-pracompact;
  \item[$(iii)$] $\mathscr{I}_\lambda^{n}$ is feebly $\omega$-bounded.
\end{itemize}
\end{theorem}

\begin{proof}
Implications $(i)\Rightarrow(iii)$ and $(ii)\Rightarrow(iii)$ are trivial.

\smallskip

$(iii)\Rightarrow(ii)$ Let $\mathscr{I}_\lambda^{n}$ be a feebly $\omega$-bounded $T_1$-semitopological semigroup. By Proposition~2 of \cite{Gutik-2017} the set $\mathscr{I}_\lambda^n\setminus\mathscr{I}_\lambda^{n-1}$ is dense in $\mathscr{I}_\lambda^n$. Fix an arbitrary infinite countable subset $D=\left\{\alpha_i\colon i\in\mathbb{N}\right\}$ in $\mathscr{I}_\lambda^n\setminus\mathscr{I}_\lambda^{n-1}$. By Lemma~2 from \cite{Gutik-2017}  every point of $D$ is isolated in $\mathscr{I}_\omega^n$, and hence by feeble $\omega$-boundedness of $\mathscr{I}_\lambda^{n}$ we get that there exists a compact subset $K\subseteq \mathscr{I}_\lambda^{n}$ such that $D\subseteq K$. Since the closure of a subset in compact space is compact, so is the closure of $D$. Hence the space $\mathscr{I}_\lambda^{n}$ is $\omega$-bounded-pracompact.

\smallskip

$(ii)\Rightarrow(i)$  Suppose the contrary: there exists a noncompact $\omega$-bounded-pracompact $T_1$-semitopological semigroup $\mathscr{I}_\lambda^{n}$. By Theorem 1 of \cite{Gutik-2017} the space $\mathscr{I}_\lambda^{n}$ is not countably compact. Then by Theorem 3.10.3 of \cite{Engelking-1989} the space $\mathscr{I}_\lambda^{n}$ has an infinite countable closed discrete subspace $D$. We put
\begin{equation*}
  A=\left\{x\in \lambda\colon x\in \operatorname{dom}\alpha\cup\operatorname{ran}\alpha \hbox{~for some~} \alpha\in D\right\}.
\end{equation*}
Since the set $D$ is countable, $\displaystyle\bigcup_{\alpha\in D}(\operatorname{dom}\alpha\cup\operatorname{ran}\alpha)$ is countable, and hence $A$ is countable, too. Then $\mathscr{I}_A^{n}$ contains $D$. By Proposition~\ref{proposition-16}\eqref{proposition-16-1}, $\mathscr{I}_A^{n}$ is a closed subspace of $\mathscr{I}_\lambda^{n}$, which implies that $D$ is an infinite countable closed discrete subspace of $\mathscr{I}_A^{n}$. This contradicts Lemma~\ref{lemma-18}, and hence $\mathscr{I}_\lambda^{n}$ is compact. \end{proof}
%%%%%%%%%%%%%%%%%%%%%%%%%%%%%%%%%%%%%%%%%%%%%%%%%%%%%%%%%%%%%%

%\section*{Acknowledgements}

%The author acknowledges Oleksandr Ravskyi for his comments and suggestions.
%%%%%%%%%%%%%%%%%%%%%%%%%%%%%%%%%%%%%%%%%%%%%%%%%%%%%%%%%%%%

%%%%%%%%%%%%%%%%%%%%%%%%%%%%%%%%%%%%%%%%%%%%%%%%%%%%%
\medskip
%\paragraph*{Acknowledgements}
%The author acknowledges Taras Banakh and the referee for useful important comments and suggestions.
%%%%%%%%%%%%%%%%%%%%%%%%%%%%%%%%%%%%%%%%%%%%%%%%%%%%%%%%%%%%

\end{document}